\documentclass[a4paper,12pt]{amsart}
\usepackage{fullpage}
\newtheorem{thm}{Theorem}
\newtheorem{prop}{Proposition}
\DeclareSymbolFont{script}{U}{eus}{m}{n}
\DeclareMathSymbol{\Wedge}{0}{script}{"5E}

\begin{document}

\title[Linearised Einstein]
{The linearised Einstein equations as a gauge theory}
\author[M.G.~Eastwood]{Michael Eastwood}
\address{\hskip-\parindent
School of Mathematical Sciences\\
University of Adelaide\\ 
SA 5005\\ 
Australia}
\email{meastwoo@gmail.com}
\begin{abstract}
We linearise the Einstein vacuum equations with a cosmological constant via the
Calabi operator from projective differential geometry.
\end{abstract}
\thanks{This work was supported by Australian Research Council (Discovery 
Project DP190102360).}

\maketitle
\section{Introduction}
This article concerns the vacuum Einstein equations with a cosmological constant
\begin{equation}\label{vacuum_einstein}R_{ab}=\lambda g_{ab},\end{equation} 
where $g_{ab}$ is a (semi-)Riemannian metric and $R_{ab}$ its Ricci tensor.  By
the contracted Bianchi identity, the smooth function $\lambda$ is obliged to be
constant.  To investigate these equations, one often linearises around the flat
metric, i.e.~considers a metric of the form
$$\eta_{ab}+\epsilon h_{ab},$$
where $\eta_{ab}$ is the Euclidean or Minkowski metric and $h_{ab}$ is
symmetric tensor.  For $h_{ab}$ fixed, this is non-singular for $\epsilon$
sufficiently small and one can compute its various curvatures, including the
Ricci tensor.  The linearised Einstein equations arise by keeping only the
linear terms in~$\epsilon$ (e.g.~\cite[\S5.7]{OT} or~\cite[\S7.5]{W}).

The aim of this article is to linearise around a general vacuum solution, 
i.e.~to consider a metric of the form 
\begin{equation}\label{perturbed_metric}
\widetilde g_{ab}=g_{ab}+\epsilon h_{ab}\end{equation}
where the `background metric' $g_{ab}$ satisfies~(\ref{vacuum_einstein}).
Writing $\bigodot^2\!\Wedge^1$ for the bundle of symmetric covariant
$2$-tensors, we shall show that there is a second order linear differential
operator
$$\textstyle{\mathcal{P}}:\bigodot^2\!\Wedge^1\to\bigodot^2\!\Wedge^1,$$
whose kernel consists of $h_{ab}$ whose corresponding perturbed metric 
(\ref{perturbed_metric}) also satisfies (\ref{vacuum_einstein}) to first order 
in~$\epsilon$. Furthermore, if ${\mathcal{K}}:\Wedge^1\to\bigodot^2\!\Wedge^1$ 
is the Killing operator $X_b\mapsto\nabla_{(a}X_{b)}$, then we shall find that 
\begin{equation}\textstyle\label{einstein_complex}
\Wedge^1\xrightarrow{\,{\mathcal{K}}\,}\bigodot^2\!\Wedge^1
\xrightarrow{\,{\mathcal{P}}\,}\bigodot^2\!\Wedge^1\end{equation} is a complex,
i.e.~the composition ${\mathcal{P}}\circ{\mathcal{K}}$ vanishes.  The range of
${\mathcal{K}}$ can be interpreted as the infinitesimal changes in metric of
the form ${\mathcal{L}}_Xg_{ab}$, where ${\mathcal{L}}_X$ denotes the Lie
derivative along a vector field~$X^a$, i.e.~the infinitesimal changes in metric
merely due to infinitesimal co\"ordinate changes (e.g.~\cite[(5.7.11)]{OT}
or~\cite[(C.2.17)]{W}).  Thus, the complex (\ref{einstein_complex}) gives a
local `potential/gauge' description of the linearised Einstein equations.

It it not at all obvious that a complex of the form (\ref{einstein_complex})
should exist.  Indeed, for an arbitrary background metric one na\"{\i}vely
might expect a complex of the form
\begin{equation}\label{BGG}
\textstyle\Wedge^1\xrightarrow{\,{\mathcal{K}}\,}\bigodot^2\!\Wedge^1
\xrightarrow{\,{\mathcal{C}}\,}
\begin{picture}(12,12)(0,2)
\put(0,0){\line(1,0){12}}
\put(0,6){\line(1,0){12}}
\put(0,12){\line(1,0){12}}
\put(0,0){\line(0,1){12}}
\put(6,0){\line(0,1){12}}
\put(12,0){\line(0,1){12}}
\end{picture}\,\end{equation}
where 
$\begin{picture}(12,12)(0,2) 
\put(0,0){\line(1,0){12}}
\put(0,6){\line(1,0){12}} 
\put(0,12){\line(1,0){12}} 
\put(0,0){\line(0,1){12}}
\put(6,0){\line(0,1){12}} 
\put(12,0){\line(0,1){12}} 
\end{picture}$ denotes the bundle of tensors $X_{abcd}$ with Riemann symmetries
$X_{abcd}=X_{[ab][cd]}$ and $X_{[abc]d}=0$ but, as we shall see, such a complex
exists only on a background of constant sectional curvature.

The notation and conventions in this article follow~\cite{OT}.  In particular,
indices are always {\em abstract\/} and do not entail any choice of local
co\"ordinates.  Round brackets on indices denote symmetrisation whilst square
brackets denote skew-symmetrisation (as in the upcoming
formula~(\ref{calabi_operator})).  We shall denote the cotangent bundle by
$\Wedge^1$ and our convention for the curvature $R_{ab}{}^c{}_d$ of a
torsion-free affine connection $\nabla_a$ is so that
\begin{equation}\label{def_curv}
(\nabla_a\nabla_b-\nabla_b\nabla_a)X^c=R_{ab}{}^c{}_dX^d,\end{equation}
for all vector fields~$X^c$.

I would like to thank Federico Costanza, Thomas Leistner, and Benjamin 
McMillan for many crucial conversations.

\section{The Calabi operator}
There is a natural candidate for the operator ${\mathcal{C}}$ in~(\ref{BGG}).
In~\cite{E} it is shown that, suitably interpreted, the {\em Killing
operator\/} $X_b\mapsto\nabla_{(a}X_{b)}$ is {\em projectively invariant\/} 
and that there is a projectively invariant operator 
${\mathcal{C}}:\bigodot^2\!\Wedge^1
\to
\begin{picture}(12,12)(0,2)
\put(0,0){\line(1,0){12}}
\put(0,6){\line(1,0){12}}
\put(0,12){\line(1,0){12}}
\put(0,0){\line(0,1){12}}
\put(6,0){\line(0,1){12}}
\put(12,0){\line(0,1){12}}
\end{picture}\,$ given by
\begin{equation}\label{calabi_operator}\textstyle h_{ab}\mapsto
\nabla_{(a}\nabla_{c)}h_{bd}-\nabla_{(b}\nabla_{c)}h_{ad}
-\nabla_{(a}\nabla_{d)}h_{bc}
+\nabla_{(b}\nabla_{d)}h_{ac}-R_{ab}{}^e{}_{[c}h_{d]e}
-R_{cd}{}^e{}_{[a}h_{b]e},\end{equation}
which we shall refer to as the {\em Calabi operator\/}.  As observed
in~\cite{CELM}, the composition ${\mathcal{C}}\circ{\mathcal{K}}$ is given by
\begin{equation}\label{C_circ_K}
X_a\longmapsto 2R_{ab}{}^e{}_{[c}\mu_{d]e}+2R_{cd}{}^e{}_{[a}\mu_{b]e}
-(\nabla^eR_{abcd})X_e,\enskip\mbox{where }\mu_{ab}\equiv\nabla_{[a}X_{b]}.
\end{equation}
In particular, this composition vanishes if and only if $g_{ab}$ has constant
sectional curvature.  Indeed, Calabi showed~\cite{C} that the complex
(\ref{BGG}) is locally exact in this case (cf.~\cite{CELM}). The Calabi 
operator therefore has the optimal symbol and it follows that there are no 
better curvature terms. 

\section{The deformation operator}
Even so, the Calabi operator is not exactly what we obtain by deforming the 
metric. The following proposition is taken from~\cite{CELM}.
\begin{prop}
If\/ $\widetilde g_{ab}=g_{ab}+\epsilon h_{ab}$, then the corresponding 
Riemann curvature tensor is
\begin{equation}\label{riemann_perturbed}\widetilde{R}_{abcd}
=R_{abcd}-\frac{\epsilon}2\left[\!\begin{array}{c}
\nabla_{(a}\nabla_{c)}h_{bd}-\nabla_{(b}\nabla_{c)}h_{ad}
-\nabla_{(a}\nabla_{d)}h_{bc}+\nabla_{(b}\nabla_{d)}h_{ac}\\[5pt]
{}+R_{ab}{}^e{}_{[c}h_{d]e}+R_{cd}{}^e{}_{[a}h_{b]e}\end{array}\!\right]
+{\mathrm{O}}(\epsilon^2).\end{equation}
\end{prop}
\begin{proof} A straightforward computation. \end{proof}
Despite the unqualified assertion in~\cite{E}, notice the opposite sign in
front of the curvature terms in comparison to~(\ref{calabi_operator}). For the 
Ricci tensor, we obtain:
\begin{prop}
If\/ $\widetilde g_{ab}=g_{ab}+\epsilon h_{ab}$, then the corresponding 
Ricci curvature tensor is
\begin{equation}\label{ricci_perturbed}\textstyle\widetilde{R}_{bd}
=R_{bd}
+\epsilon\big[R_{(b}{}^eh_{d)e}-\frac12g^{ac}({\mathcal{C}}h)_{abcd}\big]
+{\mathrm{O}}(\epsilon^2).\end{equation}
\end{prop}
\begin{proof} The inverse of the perturbed metric is
$$\widetilde g^{ab}=g^{ab}-\epsilon h^{ab}+{\mathrm{O}}(\epsilon^2),$$
where $h^{ab}\equiv g^{ac}g^{bd}h_{cd}$, in line with the usual conventions for 
`raising and lowering indices,' always with respect to the background 
metric~$g_{ab}$. We must compute $\widetilde 
R_{bd}\equiv\widetilde g^{ac}\widetilde R_{abcd}$ and for this let us write 
(\ref{riemann_perturbed}) as 
$$\textstyle\widetilde{R}_{abcd}
=R_{abcd}
-\epsilon\big[\frac12({\mathcal{C}}h)_{abcd}
+R_{ab}{}^e{}_{[c}h_{d]e}+R_{cd}{}^e{}_{[a}h_{b]e}\big]
+{\mathrm{O}}(\epsilon^2).$$
Then
$$\textstyle\widetilde R_{bd}
=R_{bd}-\epsilon\big[h^{ac}R_{abcd}+\frac12g^{ac}({\mathcal{C}}h)_{abcd}+
g^{ac}(R_{ab}{}^e{}_{[c}h_{d]e}+R_{cd}{}^e{}_{[a}h_{b]e})\big]
+{\mathrm{O}}(\epsilon^2).$$
However,
$$g^{ac}(R_{ab}{}^e{}_{[c}h_{d]e}+R_{cd}{}^e{}_{[a}h_{b]e})
=-R_{(b}{}^eh_{d)e}-h^{ac}R_{abcd}$$
and there is some cancellation to yield~(\ref{ricci_perturbed}).\end{proof}

\section{Imposing the Einstein equations}
\begin{prop}\label{first_einstein_implication}
If $g_{ab}$ satisfies \eqref{vacuum_einstein} and\/
$h_{ab}=\nabla_{(a}X_{b)}$ for some $1$-form $X_a$, then
$$g^{ac}({\mathcal{C}}h)_{abcd}=0.$$
\end{prop}
\begin{proof} If we write $\mu_{ab}\equiv\nabla_{[a}X_{b]}$, then it follows 
from (\ref{C_circ_K}) that
$$\begin{array}{rcll}g^{ac}({\mathcal{C}}h)_{abcd}
&=&g^{ac}\big[2R_{ab}{}^e{}_{[c}\mu_{d]e}+2R_{cd}{}^e{}_{[a}\mu_{b]e}
-(\nabla^eR_{abcd})X_e\big]\\[5pt]
&=&-2R_{(b}{}^e{}\mu_{d)e}-(\nabla^eR_{bd})X_e.
\end{array}$$
Now, from (\ref{vacuum_einstein}), both of these terms vanish.
\end{proof}
\begin{prop}\label{second_einstein_implication}
If $g_{ab}$ satisfies \eqref{vacuum_einstein} and\/
$\widetilde g_{ab}=g_{ab}+\epsilon h_{ab}$, then
$$\textstyle\widetilde R_{ab}=\lambda\widetilde g_{ab}
-\frac12\epsilon g^{ab}({\mathcal{C}}h)_{abcd}+{\mathrm{O}}(\epsilon^2).$$
In particular, the perturbed metric $\widetilde g_{ab}$ satisfies the vacuum 
Einstein equations $\widetilde R_{ab}=\lambda\widetilde g_{ab}$ to first order 
in $\epsilon$ if and only $g^{ac}({\mathcal{C}}h)_{abcd}=0$.\end{prop}
\begin{proof} Immediate from~(\ref{ricci_perturbed}). \end{proof}

\section{The Einstein deformation complex}
Let us define ${\mathcal{P}}:\bigodot^2\!\Wedge^1\to\bigodot^2\!\Wedge^1$ by
$$h_{ab}\longmapsto g^{ac}({\mathcal{C}}h)_{abcd},$$
equivalently, using~(\ref{def_curv}), by
\begin{equation}\label{this_is_P}h_{ab}\longmapsto
\Delta h_{bd}-2\nabla^e\nabla_{(b}h_{d)e}+\nabla_b\nabla_dh
+2R_{(b}{}^eh_{d)e},\enskip\mbox{where}\enskip\Delta\equiv 
\nabla^a\nabla_a\mbox{ and }h\equiv h^a{}_a.\end{equation}
\begin{thm} If $g_{ab}$ satisfies the vacuum Einstein
equations~\eqref{vacuum_einstein}, then \eqref{einstein_complex} is a complex.
The kernel of ${\mathcal{P}}$ consists of those symmetric tensors $h_{ab}$ such
that\/ $\widetilde g_{ab}\equiv g_{ab}+\epsilon h_{ab}$ satisfies the vacuum
Einstein equations $\widetilde R_{ab}=\lambda\widetilde g_{ab}$ to first order
in~$\epsilon$.  The range of ${\mathcal{K}}$ may be regarded as those $h_{ab}$
arising from infinitesimal changes of co\"ordinate.
\end{thm}
\begin{proof} The statements regarding ${\mathcal{P}}$ follow from
Propositions~\ref{first_einstein_implication}
and~\ref{second_einstein_implication}.  Regarding~${\mathcal{K}}$, it remains
to observe that for any $1$-form $X_a$, the Lie derivative ${\mathcal{L}}_X$
along the corresponding vector field~$X^a$, when applied to the
metric~$g_{ab}$, gives
$${\mathcal{L}}_Xg_{ab}
=X^c\nabla_cg_{ab}+(\nabla_aX^c)g_{cb}+(\nabla_bX^c)g_{ac}
=\nabla_aX_b+\nabla_bX_a=2\nabla_{(a}X_{b)},$$ 
and we are done.
\end{proof}
In the Ricci-flat case $R_{ab}=0$, the kernel of (\ref{this_is_P}) appears
as~\cite[(2.6)]{Curtis} and is formulated there in terms of spinors when the
dimension is four.  The Ricci-flat equation
$$\Delta h_{bd}-2\nabla^e\nabla_{(b}h_{d)e}+\nabla_b\nabla_dh=0$$
also appears as~\cite[(5.7.14)]{OT} and~\cite[(7.5.15)]{W} and, with non-zero Ricci
tensor, is equivalent to~\cite[(2.5)]{FH}.
	
\end{document}